\documentclass[a4paper,11pt]{article}

\usepackage{geometry}
\usepackage{amssymb, amsmath, amsthm}
\usepackage{url}
\usepackage{enumerate}
\usepackage[all,2cell]{xy}
\UseAllTwocells

\newtheorem{theorem}{Theorem}[section]
\newtheorem{proposition}[theorem]{Proposition} 
\newtheorem{lemma}[theorem]{Lemma}

\newtheorem{definition}[theorem]{Definition}
\newtheorem{remark}[theorem]{Remark}

\newcommand \NN {\mathbb N} 
\newcommand \id {\mathrm{id}} 

\newcommand \setbis[2] {\{#1\; |\; #2\}} 

\newcommand \ctg[1] {\mathbf{#1}} 
\newcommand \cat {\ctg{Cat}} 
\newcommand \grp {\ctg{Grp}} 
\newcommand \glob {\ctg{O}} 
\newcommand \ncat[1] {\text{$#1$-}{\cat}} 
\newcommand \ocat {\ncat{\infty}} 
\newcommand\wcat\ocat
\newcommand \ngrp[1] {\text{$#1$-}{\grp}} 
\newcommand \ogrp {\ngrp{\infty}} 
\newcommand \wncat {\ncat{(\infty, n)}} 
\newcommand\wgrp\ogrp

\newcommand \weakeq[1] {\Weq_{#1}} 
\newcommand \fibr[1] {\mathcal{F}_{#1}} 
\newcommand \cofibr[1] {\mathcal{C}_{#1}} 

\newcommand \IN {U} 
\newcommand \FG {F} 
\newcommand \MG {M} 
\newcommand \TT {T} 

\newcommand \I {I} 

\newcommand \Weq {\mathcal W} 
\newcommand \Imm {\mathcal Z} 

\newcommand \para \parallel

\newcommand \Comp[1] {\ast_{#1}}
\newcommand \unit[1] {1_{#1}}

\newcommand \Sce[1] {#1^\flat} 
\newcommand \Tge[1] {#1^\sharp} 
\newcommand \sce[1] {\sigma_{#1}} 
\newcommand \tge[1] {\tau_{#1}} 

\newcommand \sss[1] {\hat{\sigma}_{#1}} 
\newcommand \ttt[1] {\hat{\tau}_{#1}} 

\newcommand \inv[1] {{#1}^{-1}}
\newcommand \winv[1] {\bar{#1}}

\newcommand \eqv \sim

\newcommand \Hom[3] {#1(#2,#3)} 
\newcommand \HOM[2] {[#1,#2]} 
\newcommand \act \cdot 
\newcommand \COMP \circledast 
\newcommand \Sht[1] {{\left [#1 \right ]}} 

\newcommand \cto \curvearrowright 
\newcommand \Cto[3] {#1 \; | \; #2 \cto #3}
\newcommand \Pal[1] {#1^\natural} 
\newcommand \comp \ast 
\newcommand \Top {\pi^1}
\newcommand \Bot {\pi^2} 
\newcommand \Triv \tau 

\newcommand \cnx \Gamma 
\newcommand \Cnx[1] {\cnx(#1)}

\newcommand \lft \hat 
\newcommand \rht \tilde

\newcommand \OO[1] {\mathbf O^{#1}} 
\newcommand \DO {\partial \OO} 
\newcommand \PP[1] {\mathbf P^{#1}}

\newcommand \ii[1] {\mathbf i_{#1}}
\newcommand \jj[1] {\mathbf j_{#1}}

\newcommand \kk[1] {\mathbf k_{#1}}
\newcommand \ooo[1] {\mathbf o_{#1}} 
\newcommand \pp[1] {\mathbf p_{#1}}

\newcommand \uar {\ar@/^1pc/}
\newcommand \dar {\ar@/_1pc/}
\newcommand \Uar {\ar@/^2pc/}
\newcommand \Dar {\ar@/_2pc/}
\newcommand \ear {\ar@{-->}}
\newcommand \edar {\ar@{-->}@/_1pc/}

\newcommand \arcof {\ar@{>->}}
\newcommand \arfib {\ar@{->>}}

\newcommand \OT \leftleftarrows 

\newcommand{\doubl}[2]{\ar@<2pt>[l]^{#2}\ar@<-2pt>[l]_{#1}}

\newcommand{\doubr}[2]{\ar@<2pt>[r]^{#1}\ar@<-2pt>[r]_{#2}}
\newcommand{\doubld}[2]{\ar@<2pt>[ld]^{#2}\ar@<-2pt>[ld]_{#1}}
\newcommand{\doubd}[2]{\ar@<2pt>[d]^{#2}\ar@<-2pt>[d]_{#1}}

\makeatletter

\newcommand\DeclareMathOperatorRm[1]{%
  \expandafter\DeclareMathOperator\csname #1\endcsname{#1}}

\newcommand\DeclareMathOperatorSf[1]{%
  \expandafter\DeclareMathOperator\csname #1\endcsname{\mathsf{#1}}}

\newcommand\ndef[1]{\emph{#1}}

\DeclareMathOperator{\Coker}{Coker}
\DeclareMathOperator\A{A}

\newcommand{\oo}{$\infty$\nobreakdash}
\DeclareMathOperator\dHom{\mathsf{Hom}}
\DeclareMathOperatorSf{Aut}
\newcommand\Homens[3]{#1(#2, #3)_0}

\newcommand \Dn[1]{#1}
\newcommand{\Ths}[1]{\sigma_{#1}^{}}
\newcommand{\Tht}[1]{\tau_{#1}^{}}

\newcommand{\Glt}[2][\@empty]{%
  \ifx\@empty#1
    t_{#2}%
  \else
    t_{#2}^{#1}%
  \fi
}
\newcommand{\Gls}[2][\@empty]{%
  \ifx\@empty#1
    s_{#2}%
  \else
    s_{#2}^{#1}%
  \fi
}
\newcommand{\Glw}[2][\@empty]{%
  \ifx\@empty#1
    w_{#2}^{}%
  \else
    w_{#2}^{#1}%
  \fi
}

\newcommand\Glinv[1]{{#1}^{-1}}

\newcommand\Glid[1]{1_{#1}}

\newcommand{\Glk}[2][\@empty]{%
  \ifx\@empty#1
    k_{#2}%
  \else
    k_{#2}^{#1}%
  \fi
}

\newcommand\Wgr{\mathcal{W}_{\mathrm{gr}}}
\newcommand\Wcc{\mathcal{W}_{\mathrm{cc}}}
\newcommand\TFcc{\mathcal{T\!F}_{\mathrm{cc}}}
\newcommand\Wfolk{\mathcal{W}_{\mathrm{folk}}}
\newcommand\TFfolk{\mathcal{T\!F}_{\mathrm{folk}}}

\newcommand\Grp{\ctg{Grp}}
\newcommand{\CrC}{\ctg{CrC}}
\newcommand\CG{\mathbf{CGrp}_{\ge 2}}

\def\ufootnote{\xdef\@thefnmark{}\@footnotetext} 

\makeatother

\begin{document}

\title{The Brown-Golasi\'nski model structure\\on strict $\infty$-groupoids revisited}

\date{}

\author{\textsc{Dimitri Ara} and \textsc{Fran\c{c}ois M\'etayer}}





\maketitle

\ufootnote{\textit{2000 Mathematics Subject Classification.} 18D05, 18G55,
55U35.}
\ufootnote{\textit{Key words and phrases.} $\infty$-category, $\infty$-groupoid,
crossed complex, model category.}

\begin{abstract}
We prove that the folk model structure on strict \oo-categories transfers to
the category of strict \oo-groupoids (and more generally to the category of
strict $(\infty, n)$-categories), and that the resulting model structure on
strict \oo-groupoids coincides with the one defined by Brown and
Golasi\'nski via crossed complexes.
\end{abstract}

\section*{Introduction}\label{sec:intro}

In~\cite{browngolasinski:modscc}, Brown and Golasi\'nski build a model
structure on the category of crossed complexes. In an earlier
work~\cite{brownhiggins:grpcc}, Brown and  Higgins established an
equivalence of categories between crossed complexes and strict
\oo-groupoids, whence a model structure on the latter category. 

On the other hand, there is a ``folk'' model structure on strict
\oo-categories recently discovered by Lafont, Worytkiewicz and the second
author~\cite{lafontetal:folkms}, which extends previously known model
structures on categories~\cite{joyaltierney:strscs} and
$2$-categories~\cite{lack:quitwo}. Note that this model structure is also known as the ``natural'' or the ``categorical'' model structure by various authors.

This immediately raised the questions whether the folk model structure on
\oo-cat\-e\-go\-ries may be transferred to \oo-groupoids by inclusion, and in that
case whether the Brown-Golasi\'nski model structure may be recovered this
way.

The purpose of the present work is to show that both questions have
affirmative answers.

Our paper is organized as follows. In Section~\ref{sec:defs}, we recall the
basic definitions of strict \oo-categories and strict \oo-groupoids. Then,
in Section~\ref{sec:brown}, we describe the Brown-Golasi\'nski model
structure on crossed complexes and \oo-group\-oids. In
Section~\ref{sec:folkmsgr}, we prove the transfer theorem leading to the
definition of the folk model structure on strict \oo-groupoids. Our proof
uses neither crossed complexes nor the existence of the Brown-Golasi\'nski
model structure.  Moreover, it also applies to the category of strict
$(\infty, n)$-categories for a fixed $n$ (that is \oo-categories such that
every $m$-arrow is invertible for $m > n$). Finally, in
Section~\ref{sec:compar}, we show that the two model structures previously
defined on \oo-groupoids are in fact the same.

We thank Ronald Brown and Urs Schreiber, whose questions motivated us to go further in this work.

\section{Strict $\infty$-groupoids}\label{sec:defs}

The purpose of this section is to introduce the definitions and notations about strict
\oo-group\-oids and their weak equivalences that we will use in the sequel of this
paper. Our presentation is essentially the same as the one given in
\cite{ara:infgrp}.

\subsection{Globular sets}

Let us denote by $\glob$ the \ndef{globular category}, that is the category generated
by the graph
\[
\xymatrix{
\Dn{0} \ar@<.6ex>[r]^-{\Ths{0}} \ar@<-.6ex>[r]_-{\Tht{0}} &
\Dn{1} \ar@<.6ex>[r]^-{\Ths{1}} \ar@<-.6ex>[r]_-{\Tht{1}} & 
\cdots \ar@<.6ex>[r]^-{\Ths{i-1}} \ar@<-.6ex>[r]_-{\Tht{i-1}} &
\Dn{i} \ar@<.6ex>[r]^-{\Ths{i}} \ar@<-.6ex>[r]_-{\Tht{i}} &
\Dn{i+1} \ar@<.6ex>[r]^-{\Ths{i+1}} \ar@<-.6ex>[r]_-{\Tht{i+1}} &
\dots
}
\]
under the coglobular relations
\[\Ths{i+1}\Ths{i} = \Tht{i+1}\Ths{i}\quad\text{and}\quad\Ths{i+1}\Tht{i} =
\Tht{i+1}\Tht{i}, \qquad i \ge 0.\]

A \ndef{globular set} or \ndef{\oo-graph} is a presheaf on $\glob$. 
A globular set $X$ amounts to a diagram of sets
\[
\xymatrix{
\cdots \ar@<.6ex>[r]^-{\Gls{n+1}} \ar@<-.6ex>[r]_-{\Glt{n+1}} &
X_{n+1} \ar@<.6ex>[r]^-{\Gls{n}} \ar@<-.6ex>[r]_-{\Glt{n}} &
X_n \ar@<.6ex>[r]^-{\Gls{n-1}} \ar@<-.6ex>[r]_-{\Glt{n-1}} &
\cdots \ar@<.6ex>[r]^-{\Gls{1}} \ar@<-.6ex>[r]_-{\Glt{1}} &
X_1 \ar@<.6ex>[r]^-{\Gls{0}} \ar@<-.6ex>[r]_-{\Glt{0}} &
X_0
}
\]
satisfying the globular relations
\[\Gls{i}\Gls{i+1} = \Gls{i}\Glt{i+1}\quad\text{and}\quad\Glt{i}\Gls{i+1} =
\Glt{i}\Glt{i+1}, \qquad i \ge 0.\]

For $i \ge j \ge 0$, we will denote by $\Gls[i]{j}$ and $\Glt[i]{j}$ the
maps from $X_i$ to $X_j$ defined by
\[\Gls[i]{j} = \Gls{j}\cdots\Gls{i-2}\Gls{i-1}\quad\text{and}\quad
  \Glt[i]{j} = \Glt{j}\cdots\Glt{i-2}\Glt{i-1}.\]
If $X$ is a globular set, we will call $X_0$ the set of \ndef{objects} of
$X$ and $X_n$ for $n \ge 0$ the set of \ndef{$n$-arrows} or
\ndef{$n$-cells}. 
The notation $u \colon x \to y$ will mean that $u$ is an
$n$-arrow for $n \ge 1$ whose source is an $(n-1)$-arrow $x$ (that is
$\Gls{n-1}(u) = x$) and whose target is an $(n-1)$-arrow $y$ (that is
$\Glt{n-1}(u) = y$). 
We will say that two $n$-arrows $u$ and $v$ are \ndef{parallel} if either $n = 0$,
or $n \ge 1$ and $u, v$ have same source and same target.
For $i \ge j \ge 0$, if $u$ is an $i$-arrow, we will often write
$\Gls{j}(u)$ for $\Gls[i]{j}(u)$ and similarly
$\Glt{j}(u)$ for $\Glt[i]{j}(u)$. 

If $u$ and $v$ are $n$-arrows, $X(u, v)$ will
denote the globular set whose $k$-arrows are the $(n+k+1)$-arrows $a$ of $G$ such that
$\Gls{n}(a) = u$ and $\Glt{n}(a) = v$. In particular, $X(u, v)_0$ is the set
of $(n+1)$-arrows $a \colon u \to v$ in $X$.

\subsection{Strict $\infty$-categories}

An \ndef{\oo-precategory} is a globular set $C$ endowed with maps
\[
  \begin{split}
  \comp_j^i & \colon (X_i, \Gls[i]{j}) \times_{X_j} (\Glt[i]{j}, X_i) \to 
      X_i,\quad i > j \ge 0,
      \\
  \Glk{i} & \colon X_{i-1} \to X_i, \quad i \ge 1,
  \end{split}
\]
such that
\begin{enumerate}
    \item 
      for every
      $(u, v)$ in $(X_i, \Gls[i]{j}) \times_{X_j} (\Glt[i]{j}, X_i)$ with
      $i > j \ge 0$, we have
  \[
  \Gls{i-1}(u \comp_j^i v) = 
  \begin{cases}
    \Gls{i-1}(v), & j = i - 1 \\
    \Gls{i-1}(u) \comp_j^{i-1} \Gls{i-1}(v), & j < i - 1
  \end{cases}\text{;}
  \]
    \item  
      for every
      $(u, v)$ in $(X_i, \Gls[i]{j}) \times_{X_j} (\Glt[i]{j}, X_i)$ with
      $i > j \ge 0$, we have
\[
  \Glt{i-1}(u \comp_j^i v) = 
  \begin{cases}
    \Glt{i-1}(u), & j = i - 1 \\
    \Glt{i-1}(u) \comp_j^{i-1} \Glt{i-1}(v), & j < i - 1
  \end{cases}\text{;}
  \]
  \item for every $u$ in $X_i$ with $i \ge 0$, we have
  \[
  \Gls{i}\Glk{i+1}(u) = u = \Glt{i}\Glk{i+1}(u).
  \] 
\end{enumerate}

For $i \ge j \ge 0$, we will denote by $\Glk[j]{i}$ the map from $X_j \to
X_i$ defined by
  \[ \Glk[j]{i} = \Glk{i}\cdots\Glk{j+2}\Glk{j+1}. \]
If $u$ and $v$ are $n$-arrows for $n \ge 1$ of an \oo-precategory, we will
often write $u \comp_k v$ for $u \comp^n_k v$. If $u$ is an
$n$-arrow, we will often write $\Glid{u}$ for the iterated identity $\Glk[n]{m}(u)$
in a dimension $m \ge n$ clear by the context.

\begin{definition}
A \ndef{strict \oo-category} is an \oo-precategory $X$ such that the following axioms are
satisfied:
\begin{enumerate}
  \item Associativity \\
     for every $(u, v, w)$ in
    $(X_i, \Gls[i]{j}) \times_{X_j} (\Glt[i]{j}, X_i,
    \Gls[i]{j}) \times_{X_j} (\Glt[i]{j}, X_i)$ with $i > j \ge 0$,
    we have
    \[ (u \comp_j v) \comp_j w = u \comp_j (v \comp_j w)\text{;} \]
  \item Exchange law \\
   for every $(u, u', v, v')$ in
    \[ (X_i, \Gls[i]{j}) \times_{X_j} (\Glt[i]{j}, X_i, \Gls[i]{k})
    \times_{X_k} (\Glt[i]{k}, X_i, \Gls[i]{j}) \times_{X_j} (\Glt[i]{j},
    X_i),\]
    with $i > j > k \ge 0$, we have
    \[ (u \comp_j u') \comp_k (v \comp_j v') = (u \comp_k v)
    \comp_j ( u' \comp_k v')\text{;} \]
  \item Units \\
   for every $u$ in $X_i$ with $i \ge 1$ and every $j$ such that $i > j \ge
   0$, we have
    \[ u \comp_j \Glid{\Gls{j}(u)} = u = \Glid{\Glt{j}(u)} \comp_j u\text{;} \]
  \item Functoriality of units \\
    for every
    $(u, v)$ in $(X_i, \Gls[i]{j}) \times_{X_j} (\Glt[i]{j} ,X_i)$
    with $i > j \ge 0$, we have the following equality between
    $(i+1)$-arrows:
    \[ \Glid{u \comp_j v} = \Glid{u} \comp_j \Glid{v}. \]
\end{enumerate}

A \ndef{morphism of strict \oo-categories} or \ndef{\oo-functor} is a morphism
of globular sets compatible with the maps $\comp^i_j$ and $\Glk{i}$.
\end{definition}

We will denote by $\wcat$ the category of \oo-categories. 
This category is a full reflexive subcategory of the presheaf category of
globular sets. Moreover, it is stable under filtered colimits.  Hence, by
Theorem 1.46 of \cite{adamekrosicky:locpac}, $\wcat$ is locally presentable.

Note that if $u$ and $v$ are two $n$-arrows of a strict \oo-category $C$,
the globular set $C(u, v)$ inherits a structure of strict \oo-category.

\subsection{Strict $\infty$-groupoids}

Let $C$ be a strict \oo-category and $u$ an $i$-arrow for $i \ge 1$. For $j$
such that $0 \le j < i$, a \ndef{$\comp^i_j$-inverse} $v$ of $u$ is an
$i$-arrow such that $\Gls{j}(v) = \Glt{j}(u)$ and $\Glt{j}(v) = \Gls{j}(u)$, satisfying
\[
u \comp_j v = \Glid{\Glt{j}(u)}
\quad\text{and}\quad
v \comp_j u = \Glid{\Gls{j}(u)}.
\]
It is easy to see that if it exists, such an inverse is unique.
For $i > j \ge 0$, we will say that $C$ \ndef{admits $\comp^i_j$-inverses}
if every $i$-arrow of $C$ admits a $\comp^i_j$-inverse.

\begin{definition}
A \ndef{strict \oo-groupoid} is a strict \oo-category which admits
$\comp^i_j$-inverses for every $i > j \ge 0$.  We will denote by $\wgrp$ the
full subcategory of $\wcat$ whose objects are strict \oo-groupoids.

Let $n \ge 0$. A \ndef{strict $(\infty, n)$-category} is a
strict \oo-category which admits $\comp^i_j$-inverses for every $i > j \ge
n$.  We will denote by $\wncat$ the full subcategory of $\wcat$ whose
objects are strict $(\infty, n)$-categories. Note that for $n = 0$ we
recover the category of strict \oo-groupoids.
\end{definition}

The same argument as for $\wcat$ shows that $\wgrp$ is a locally presentable
category.

If $G$ is a strict \oo-groupoid and $u$ is an $i$-arrow of $G$ for $i \ge
1$, we will denote by $\Glw[i]{j}(u)$ or simply by $\Glw{j}(u)$ the
$\comp^i_j$-inverse of $u$ and by $\Glinv{u}$ the $\comp^i_{i-1}$-inverse.
Note that if $u$ and $v$ are two $n$-arrows of a strict \oo-groupoid $G$,
the strict \oo-category $G(u, v)$ is a strict \oo-groupoid.

\begin{proposition}
Let $C$ be a strict \oo-category. Then the following assertions are
equivalent:
\begin{enumerate}
\item $C$ is a strict \oo-groupoid;
\item $C$ admits $\comp^i_{i-1}$-inverses for every $i \ge 1$;
\item $C$ admits $\comp^i_0$-inverses for every $i \ge 1$;
\item for all $i \ge 1$, there exists $j$ satisfying $0 \le j < i$ such that
$C$ admits $\comp^i_j$-inverses.
\end{enumerate}
\end{proposition}

\begin{proof}
By induction, it suffices to show that for every $i > j > k \ge 0$, if $C$ admits
$\comp^j_k$-inverses, then $C$ admits $\comp^i_k$-inverses if and only if it
admits $\comp^i_j$-inverses. By using the fact that the $2$-graph
\[
\xymatrix{
C_i \ar@<.6ex>[r]^-{\Gls[i]{j}} \ar@<-.6ex>[r]_-{\Glt[i]{j}} &
C_j \ar@<.6ex>[r]^-{\Gls[j]{k}} \ar@<-.6ex>[r]_-{\Glt[j]{k}} & 
C_k
}
\]
has a natural structure of $2$-category, one can assume that $k = 0$, $j =
1$ and $i = 2$. The result is thus a consequence of the following lemma.
\end{proof}

\begin{lemma}
Let $C$ be a $2$-category whose $1$-arrows are invertible. Then a $2$-arrow is
invertible for horizontal composition (i.e., $\comp^2_0$) if and only if it is invertible
for vertical composition (i.e., $\comp^2_1$).
\end{lemma}

\begin{proof}
Let $a \colon u \to v$ be a $2$-arrow. Suppose $a$ admits a horizontal inverse
$a^\ast$. Then $v \comp_0 a^\ast \comp_0 u$ is a vertical
inverse. Conversely, suppose $a$ admits a vertical inverse~$a^{-1}$. Then
$v^{-1} \comp_0 a^{-1} \comp_0 u^{-1}$ is a horizontal inverse.
\end{proof}

\subsection{Weak equivalences of strict $\infty$-groupoids}

Let $G$ be a strict \oo-groupoid. An $n$-arrow $u$ of $G$ is
\ndef{homotopic} to another $n$-arrow $v$ if there exists an $(n+1)$-arrow
from $u$ to $v$. This implies that the arrows $u$ and $v$ are parallel. If
$u$ is homotopic to $v$, we will write $u \sim v$. The relation $\sim$ is an
equivalence relation on $G_n$: the properties with respect to source and
target of the maps $\Glk{n+1}$,
$\Glw[n]{n-1}$ and $\comp^n_{n-1}$ imply respectively that $\sim$ is
reflexive, symmetric and transitive. 

Let us denote by $\overline{G_n}$ the quotient of $G_n$ by $\sim$. The
composition
\[ \comp^{n}_{n-1} \colon G_n
\times_{G_{n-1}} G_n \to G_n
\]
induces a map
\[
\comp^{n}_{n-1} \colon \overline{G_n} \times_{G_{n-1}} \overline{G_n} \to
\overline{G_n},
\]
thanks to the properties with respect to source and target of the composition
$\comp^{n+1}_{n-1}$. For $n \ge 1$, we can thus define a groupoid
$\varpi_n(G)$ whose objects are $(n-1)$-arrows of $G$ and whose morphisms
are elements of $\overline{G_n}$. It is clear that $\varpi_n$ defines a
functor from the category of strict \oo-groupoids to the category of groupoids.

\begin{definition}
The set of \ndef{connected components} of $G$ is
\[ \pi_0(G) = \pi_0(\varpi_1(G)) = \overline{G_0}. \]
For $n \ge 1$ and $x$ an object of $G$, the \ndef{$n$-th homotopy group of
$G$ at $x$} is
\[ \pi_n(G, x) = \pi_1(\varpi_n(G), \Glid{x}) =
\Aut_{\varpi_n(G)}(\Glid{x}). \]
\end{definition}

By functoriality of the $\varpi_n$'s, $\pi_0$ induces a functor from the
category of strict \oo-groupoids to the category of sets, and $\pi_n$, for $n \ge 1$,
induces a functor from the category of pointed strict \oo-groupoids to the category
of groups. By the Eckmann-Hilton argument, the groups $\pi_n(G,
x)$ are abelian for $n \ge 2$.
More generally, if $u$ and $v$ are two $(n-1)$-arrows for $n \ge 1$ we set
\[
\pi_n(G, u, v) = \dHom_{\varpi_n(G)}(u, v)
\quad\text{and}\quad
\pi_n(G, u) = \pi_n(G, u, u).
\]

\begin{definition}
A morphism $f \colon G \to H$ of strict \oo-groupoids is a \ndef{weak equivalence
of strict \oo-groupoids} if the map $\pi_0(f) \colon \pi_0(G) \to \pi_0(H)$ is a
bijection and if for all $n \ge 1$ and all object $x$ of $G$, the morphism
$\pi_n(f, x) \colon \pi_n(G, x) \to \pi_n(H, f(x))$ is a group isomorphism.
We will denote by $\Wgr$ the class of such weak equivalences.
\end{definition}

\begin{proposition}\label{prop:def_w_eq}
Let $f \colon G \to H$ be a morphism of strict \oo-groupoids. The following conditions are
equivalent:
\begin{enumerate}
\item\label{item:w_eq} $f$ is a weak equivalence of strict \oo-groupoids;
\item $\pi_0(f) \colon \pi_0(G) \to \pi_0(H)$ is a bijection and for all $n \ge
1$ and every $(n-1)$-arrow $u$ of~$G$, $f$ induces a bijection
\[ \pi_n(G, u) \to \pi_n(H, f(u))\text{;} \]
\item $\varpi_1(f) \colon \varpi_1(G) \to \varpi_1(H)$ is an equivalence of categories and
for all $n \ge 2$ and every pair $(u, v)$ of parallel $(n-1)$-arrows
  of~$G$, $f$
 induces a bijection
 \[ \pi_n(G, u, v) \to \pi_n(H, f(u), f(v))\text{;} \]
\item\label{item:folk_w_eq} $\varpi_1(f) \colon \varpi_1(G) \to \varpi_1(H)$ is full and essentially surjective,
and for all $n \ge 2$ and every pair $(u, v)$ of parallel $(n-1)$-arrows
of $G$, $f$ induces a surjection
\[ \pi_n(G, u, v) \to \pi_n(H, f(u), f(v))\text{.} \]
\end{enumerate}
\end{proposition}

\begin{proof}
$1 \Rightarrow 2$)
The case $n = 1$ is obvious. Let $n \ge 2$ and let $u$ be an $(n-1)$-arrow
of~$G$. Set $x = \Gls{0}(u)$. 
The map
\[ \pi_{n}(G,x) \to \pi_n(G, u) \]
which sends an $n$-arrow $a \colon \Glid{x} \to \Glid{x}$
to the $n$-arrow $1_u \comp_0 a  \colon u \to u$, is an isomorphism.
Moreover $f$ commutes with this isomorphism, that is the square
\[
\xymatrix{
\pi_n(G, x) \ar[d] \ar[r] & \pi_n(G, u) \ar[d] \\
\pi_n(H, f(x))  \ar[r] & \pi_n(H, f(u)) \\
}
\]
is commutative. The map $\pi_n(G, u) \to \pi_n(H, f(u))$ is thus a bijection
for $n \ge 2$.

$2 \Rightarrow 3$) 
Let $n \ge 1$ and let $u, v$ be two parallel $(n-1)$-arrows of $G$. Suppose there
exists an $n$-arrow $a \colon u \to v$ in $G$. The map 
\[ \pi_n(G, u) \to \pi_n(G, u, v) \]
which sends an $n$-arrow $b \colon u \to u$ to the $n$-arrow
$a \comp_{n-1} b \colon u \to v$, is a bijection. Moreover $f$ commutes with this
bijection, that is the square
\[
\xymatrix{
\pi_n(G, u) \ar[d] \ar[r] & \pi_n(G, u, v) \ar[d] \\
\pi_n(H, f(u))  \ar[r] & \pi_n(H, f(u), f(v)) \\
}
\]
is commutative.

Thus to conclude it suffices to show that if there exists an $n$-arrow
$b \colon f(u) \to f(v)$ in $H$, then there exists an $n$-arrow $a \colon u \to v$ in
$G$. It is clear when $n = 1$ by injectivity of $\pi_0(f)$. Let $n \ge 2$
and let $b \colon f(u) \to f(v)$ be an $n$-arrow of $H$. Set $x = \Gls{n-2}(u)$. The
arrow $\Glid{\Glinv{f(u)}} \comp_{n-2} b$  is an $n$-arrow of $H$
from $\Glid{f(x)} \colon f(x) \to f(x)$ to 
$\Glinv{f(u)} \comp_{n-2} f(v) \colon f(x) \to f(x)$. Since the map
\[ \pi_{n-1}(G, x) \to \pi_{n-1}(H, f(x)) \]
is injective, there exists an $n$-arrow $a'$ of $G$ from
$\Glid{x}$ to  $\Glinv{u} \comp_{n-2} v$.
Then $a = \Glid{u} \comp_{n-2} a'$ is an $n$-arrow of $G$ from
$u$ to $v$.

$3 \Rightarrow 1$) Obvious.

$4 \Rightarrow 3$)
Let $n \ge 1$, let $u,v$ be two parallel $(n-1)$-arrows of $G$ and let $a, b$
be two $n$-arrows from $u$ to $v$.
Suppose we have $f(a) = f(b)$ in $\pi_n(H, f(u), f(v))$. Then there exists
an $(n+1)$-arrow of $H$ from $f(a)$ to $f(b)$. By
surjectivity of the map
\[ \pi_{n+1}(G, a, b) \to \pi_{n+1}(H, f(a), f(b)), \]
there exists an $(n+1)$-arrow in $G$ from $a$ to $b$.
Thus $a = b$ in $\pi_n(G, u, v)$.

$3 \Rightarrow 4$) Obvious.

\end{proof}

\section{The Brown-Golasi\'nski model structure}\label{sec:brown}

In \cite{browngolasinski:modscc}, Brown and Golasi\'nski introduce a model
category structure on the category of crossed complexes. By the equivalence
of categories between crossed complexes and strict \oo-groupoids constructed
in \cite{brownhiggins:grpcc}, this model structure induces a model structure
on strict \oo-groupoids. The purpose of this section is to describe this
model structure.

\subsection{Crossed complexes}

Let us denote by $\Grp$ the category of groups and by
$\CG$ the category of homological complexes of (not necessarily
commutative) groups in dimension greater or equal to $2$, that is of
sequences of morphisms of groups
\[ \dots \to C_n \xrightarrow{d_n} C_{n-1} \to \dots \to C_3 \xrightarrow{d_3} C_2 \]
such that for every $n \ge 4$, we have $d_{n-1}d_n = 1$, where $1$ denotes
the unit element of $C_{n-2}$. We have an inclusion
functor $i_2 \colon \Grp \to \CG$ which sends a group $G$ to the complex
concentrated in degree $2$ on $G$.

Let $C_{\le 1}$ be a groupoid. We will denote by $C_0$ its set of objects
and by $C_1(x, y)$ the set of morphisms from an object $x$ to an object $y$ in $C_{\le 1}$.
Let $C_1 \colon C_{\le 1} \to \Grp$ be the functor defined in the following way:
an object $x$ of $C_{\le 1}$ is sent to the group $C_1(x) = C_1(x, x)$; 
a morphism $u \colon x \to y$ of $C_{\le 1}$ is sent to the morphism of
conjugation by $u$, i.e.,
\[
\begin{split}
C_1(x) & \to C_1(y) \\
v & \mapsto uvu^{-1}.
\end{split}
\]

\begin{definition}
A \ndef{precrossed complex} consists of
\begin{itemize}
\item a groupoid $C_{\le 1}\text{;}$
\item a functor $C_{\ge 2} \colon C_{\le 1} \to \CG\text{;}$
\item an augmentation of $C_{\ge 2}$ over $C_1$, that is a natural
transformation $d_2 \colon C_{\ge 2} \to i_2C_1$.
\end{itemize}
\end{definition}

\noindent
Explicitly, a precrossed complex is given by 
\begin{itemize}
\item for every $x$ in $C_0$, a complex
\[ C_{\ge 2}(x) = \dots \to C_n(x) \xrightarrow{d_n} C_{n-1}(x) \to \dots \to C_3(x)
\xrightarrow{d_3} C_2(x)\text{;} \]
\item for every $x$ in $C_0$, a morphism $d_2 \colon C_2(x) \to C_1(x)\text{;}$
\item for every $n \ge 2$ and every morphism $u \colon x \to y$ of $C_{\le 1}$, a morphism $C_n(x) \to C_n(y)$
of groups
functorial in $u$,
\end{itemize}
such that for every morphism $u \colon x \to y$ of $C_{\le 1}$, the diagram
\[
\xymatrix{
  \ar[r] & C_n(x) \ar[d] \ar[r]^-{d_n} & C_{n-1}(x) \ar[d]
  \ar[r] & \cdots \ar[r] & C_2(x) \ar[r]^{d_2} \ar[d] & C_1(x) \ar[d] \\
  \ar[r] & C_n(y) \ar[r]^-{d_n} & C_{n-1}(y) \ar[r] & \cdots
  \ar[r] & C_2(y) \ar[r]^{d_2} & C_1(y)\text{,} \\
}
\]
where $C_1(x) \to C_1(y)$ is the conjugation by $u$, is commutative.

If $C$ is a precrossed complex we will denote by $C(x)$ the augmented
complex 
\[\dots \to C_n(x) \xrightarrow{d_n} C_{n-1}(x) \to \dots \to C_3(x)
\xrightarrow{d_3} C_2(x) \xrightarrow{d_2} C_1(x) \text{.} \]
If $u \colon x \to y$ is a morphism of $C_{\le 1}$ and $n \ge 1$, we will call
the map $C_n(x) \to C_n(y)$ the \ndef{action} of $u$ on $C_n(x)$. In
particular, an element of $C_1(x)$ \ndef{acts} on $C_n(x)$ for all $n \ge
1$.

\begin{definition}
A precrossed complex $C$ is a \ndef{crossed complex} if for all $x$ in
$C_0$ the following conditions are satisfied:
\begin{itemize}
\item for every $n \ge 3$, the group $C_n(x)$ is abelian;
\item for every $u$ in $C_2(x)$, the element
$d_2(u)$ of $C_1(x)$ acts
\begin{itemize}
\item by conjugation by $u$ on $C_2(x)\text{;}$
\item trivially on $C_n(x)$ for $n \ge 3$.
\end{itemize}
\end{itemize}
\end{definition}

\begin{definition}
Let $C$ and $D$ be two crossed complexes.
A \ndef{morphism of crossed complexes} $f \colon C \to D$ consists of
\begin{itemize}
\item a functor $f_{\le 1}\colon C_{\le 1} \to D_{\le 1}\text{;}$
\item a natural transformation $f_{\ge 2} \colon C_{\ge 2} \to D_{\ge
2}f_{\le 1}\text{,}$
\end{itemize}
such that $f_{\ge 2}$ is compatible with the augmentation, i.e.,
the square
\[
\xymatrix{
  C_{\ge 2} \ar[d]_{d_2} \ar[r] & D_{\ge 2}f_{\le 1} \ar[d]^{d_2 f_{\le 1}} \\
  i_2C_1 \ar[r] & i_2D_1f_{\le 1} \\
}
\]
is commutative.
\end{definition}

\noindent
Explicitly, a morphism $f \colon C \to D$ is given by
\begin{itemize}
\item a functor $f_{\le 1}\colon C_{\le 1} \to D_{\le 1}\text{;}$
\item for every $n \ge 2$ and every $x$ in $C_0$, a group morphism
$f_n \colon C_n(x) \to D_n(f_0(x))$, where $f_0$ denotes the restriction of $f_{\le 1}$
to objects,
\end{itemize}
such that for every $x$ in $C_0$, the diagram
\[
\xymatrix{
  \ar[r] & C_n(x) \ar[d]_{f_n} \ar[r]^-{d_n} & C_{n-1}(x) \ar[d]_{f_{n-1}}
  \ar[r] & \cdots \ar[r] & C_1(x) \ar[d]_{f_1} \\
  \ar[r] & D_n(f_0(x)) \ar[r]^-{d_n} & D_{n-1}(f_0(x)) \ar[r] & \cdots
  \ar[r] & D_1(f_0(x)) \text{,} \\
}
\]
where $f_1$ denotes the restriction of $f_{\le 1}$ to $C_1(x)$, is
commutative.
We will often simply write $f$ for $f_n$.

We will denote the category of crossed complexes by $\CrC$ .

\subsection{The equivalence with strict $\infty$-groupoids}\label{sec:equiv_grp_cc}
Let $G$ be a strict \oo-groupoid. One can associate to $G$ a precrossed complex $C$ in the
following way:
\begin{itemize}
\item the groupoid $C_{\le 1}$ is the $1$-truncation of $G$ 
obtained from $G$ by throwing out the $n$-arrows for $n \ge 2\text{;}$
\item $C_n(x)$ is the set of $n$-arrows $u$ of $G$ whose source is an
iterated unit of an object, that is such that there exists an object $x$ of
$G$ such that $\Gls{n-1}(u) = \Glid{x}\text{;}$
\item the group law on $C_n(x)$ is induced by the composition 
 $\comp_0^n \colon G_n \times_{G_0} G_n \to G_n\text{;}$
\item $d_n \colon C_n(x) \to C_{n-1}(x)$ is induced by the target map $\Glt{n-1} \colon
G_n \to G_{n-1}\text{;}$
\item if $u \colon x \to y$ is a morphism in $C_{\le 1}$ and $v$ a morphism in
$C_n(x)$ for $n \ge 2$, the action of $u$ on $v$ is
\[ u \comp_0 v \comp_0 \Glw{0}(u). \]
\end{itemize}
This precrossed complex is a crossed complex (see paragraph 3 of
\cite{brownhiggins:grpcc}). Moreover this construction defines a functor $\A
\colon \wgrp \to \CrC$.

\begin{theorem}[Brown-Higgins]
The functor $\A \colon \wgrp \to \CrC$ is an equivalence of categories.
\end{theorem}

\begin{proof}
See Theorem 4.1 of \cite{brownhiggins:grpcc}.
\end{proof}

\subsection{The model structure on crossed complexes}

\smallbreak
 
\begin{definition}
Let $C$ be a crossed complex. The set of \ndef{connected components} of $C$
is
\[
 \pi_0(C) = \pi_0(C_{\le 1}).
\]
For $x$ in $C_0$, the \ndef{fundamental group of $C$ at $x$} is
\[ \pi_1(C, x) = \Coker(d_2 \colon C_2(x) \to C_1(x)), \]
and for $n \ge 2$, the \ndef{$n$-th homotopy group of $C$ at $x$} is
\[ \pi_n(C, x) = H_n(C(x)). \]
\end{definition}

It is clear that $\pi_0$ defines a functor from the category of crossed
complexes to the category of sets and that for $n \ge 1$, $\pi_n$ defines a
functor from the category of pointed crossed complexes to the category of
groups.

\begin{definition}
A morphism $f \colon C \to D$ of crossed complexes is a \ndef{weak equivalence of
crossed complexes}
(see \cite{browngolasinski:modscc})
if the map $\pi_0(f) \colon \pi_0(C) \to \pi_0(D)$ is a bijection and if for every
$x$ in $C_0$ and every $n \ge 1$, the morphism $\pi_n(f, x) \colon \pi_n(C, x)
\to \pi_n(D, f(x))$ is an isomorphism.
\end{definition}

\begin{definition}
Let $f \colon C \to D$ be a morphism of crossed complexes. The morphism $f$ is a
\ndef{trivial fibration of crossed complexes} (see Proposition 2.2 (ii) of
\cite{browngolasinski:modscc}) if the following conditions are satisfied:
\begin{itemize}
\item for every $y$ in $D_0$, there exists $x$ in $C_0$ such that $f(x) =
y$\text{;}
\item for every $x$, $y$ in $C_0$ and every morphism $v \colon f(x) \to
f(y)$ in $D_{\le 1}$, there exists a morphism $u \colon x \to y$ in $C_{\le
1}$ such that $f(u) = v$\text{;}
\item  for every $n \ge 2$, $x$ in $C_0$, $t$ in $C_{n-1}(x)$ and every
$v$ in $D_n(f(x))$ such that $d_n(v) = f(t)$, there exists
$u$ in $C_n(x)$ such that $d_n(u) = t$ and $f(u) = v$.
\end{itemize}
\end{definition}

\begin{theorem}[Brown-Golasi\'nski]
The weak equivalences and trivial fibrations of crossed complexes define a model
category structure on the category of crossed complexes.
\end{theorem}

\begin{proof}
See Theorem 2.12 of \cite{browngolasinski:modscc}.
\end{proof}

\subsection{The model structure on strict $\infty$-groupoids}

One obtains a model category structure on strict \oo-groupoids
by transferring the model structure on crossed complexes defined in
the previous section via the equivalence of categories $\A \colon \wgrp \to
\CrC$. We will call this model structure the Brown-Golasi\'nski model
structure. A morphism $f$ of strict \oo-groupoids is a weak equivalence
(respectively a trivial fibration) for the Brown-Golasi\'nski model structure
if and only if $\A(f)$ is a weak equivalence (respectively a trivial fibration) of
crossed complexes. We will denote these classes by $\Wcc$ and $\TFcc$
respectively.

We now describe these two classes more explicitly.

\begin{proposition}\label{prop:cc_w_eq}
We have $\Wcc = \Wgr$.
In other words, a morphism $f
\colon G \to H$ 
of strict \oo-group\-oids 
is a weak equivalence of strict \oo-groupoids if and only if
the morphism $\A(f) \colon \A(G) \to \A(H)$ is a weak equivalence of crossed
complexes.
\end{proposition}

\begin{proof}
Since the two notions of weak equivalences are defined in terms of homotopy
groups, it suffices to show that the two notions of homotopy groups
coincide.

Let $G$ be a groupoid.
By definition, $\pi_0(G) = \pi_0(\A(G))$. 

Let $x$ be an object of $G$.  By definition, $\pi_1(\A(G), x)$ is the
cokernel of the morphism $\A(G)_2(x) \to \A(G)_1(x)$.  The set $\A(G)_1(x)$
is the set of $1$-arrows $u \colon x \to x$ in $G$ and two such arrows $u,
u'$ are identified in the cokernel if and only if there exists a $2$-arrow
from $\Glid{x}$ to $u \comp_0 \Glinv{u'}$ in $G$. This condition is
equivalent to the existence of a $2$-arrow from $u$ to $u'$. Hence $\pi_1(G,
x) = \pi_1(\A(G), x)$.

Let $n \ge 2$. The kernel of the map $d_n \colon \A(G)_n(x) \to \A(G)_{n-1}(x)$ is the
set $G_n(x, x)$. Thus the same argument as in dimension $1$ shows that
$\pi_n(G, x) = \pi_n(\A(G), x)$.
\end{proof}

\begin{proposition}\label{prop:cc_tr_fib}
A morphism $f \colon G \to H$ of strict \oo-groupoids is in $\TFcc$ if and only
if it satisfies the following conditions:
\begin{itemize}
\item for every object $y$ of $H$, there exists an object $x$ of $G$ such
that $f(x) = y$;
\item for every pair $(x, y)$ of objects of $G$, the map 
\[ \Homens{G}{x}{y} \to \Homens{G}{f(x)}{f(y)} \]
 induced by $f$ is a surjection;
\item for all $n \ge 2$, every object $x$ of $G$ and every $(n-1)$-arrow  $u
\colon \Glid{x} \to \Glid{x}$, the map
\[ \Homens{G}{\Glid{x}}{u} \to \Homens{H}{\Glid{f(x)}}{f(u)} \]
is surjective.
\end{itemize}
\end{proposition}

\begin{proof}
By definition, $f$ is in $\TFcc$ if and only if $\A(f)$ is a trivial
fibration of crossed complexes.  This proposition is then just a matter of
translation using the definition of the functor $\A$.
\end{proof}

\section{The folk model structure on $\ogrp$}\label{sec:folkmsgr}

This section shows that the folk model structure on $\ocat$ defined in~\cite{lafontetal:folkms} transfers to $\ogrp$ via the inclusion functor
\begin{displaymath}
  \IN\colon\ogrp\to\ocat.
\end{displaymath}
We first give a brief review of the main results of~\cite{lafontetal:folkms}, and introduce the material needed to prove the transfer theorem.

\subsection{The folk model structure on $\ocat$}\label{subsec:folkmscat}

Given an $\infty$-category $C$, we define {\em reversible} cells in $C$ and
the relation of {\em $\omega$\nobreakdash-equival\-ence} between cells of
$C$ by mutual coinduction on $n\geq0$.
\begin{definition}\label{def:omeq}
Let $n\in\NN$:
\begin{itemize}
\item an $n$-cell $x$ is {\em $\omega$-equivalent} to an $n$-cell $y$ if
there is a reversible $(n+1)$-cell $u\colon x\to y$;
\item an $(n+1)$-cell $u\colon x\to y$ is {\em reversible} if there is an $(n+1)$-cell $\winv u\colon y\to x$ such that $\winv u\Comp n  u$ is $\omega$-equivalent to $\unit x$ and $u\Comp n \winv u$ is $\omega$-equivalent to $\unit y$.
\end{itemize}
\end{definition}
Note that, for each $r\geq 0$, if two cells are $r$-equivalent in the sense
of~\cite{street:algors}, then they are $\omega$-equivalent, with the
converse being false. We also refer to~\cite{jacobsrutten:coindu} for a
gentle introduction to coinductive methods. Remark also that if $G$ is an
$\infty$-groupoid, any $(n+1)$-cell $u$ of $G$ is reversible and the cell
$\winv u$ whose existence is stated in Definition~\ref{def:omeq} is of
course just $\inv u$.

Let $\Weq$ denote the class of those morphisms $f\colon C\to D$ satisfying the following two conditions:
  \begin{enumerate}
  \item for each $0$-cell $y$ in $D$, there is a $0$-cell $x$ in $C$ such that $fx$ is $\omega$-equivalent to $y$;
  \item for each pair $(x, x')$ of parallel $n$-cells in $C$, where $n\geq0$, and each $(n+1)$-cell $v\colon fx\to fx'$, there is an $(n+1)$-cell $u\colon x\to x'$ such that $fu$ is $\omega$-equivalent to~$v$.
  \end{enumerate}
Now for each $n\geq 0$, we define the {\em $n$-globe} $\OO n$ as the free $\infty$-category generated by the representable globular set $Y(n)=\Hom{\glob}{n}{-}$. Thus $Y(n)$ has exactly one $n$-cell, two $k$-cells for each $k<n$ and no $k$-cell for $k>n$. Let also $\partial Y(n)$ be the globular set having the same cells as $Y(n)$ except in dimension $n$ where $(\partial Y(n))_n=\emptyset$. We denote by $\DO n$ the free $\infty$-category on $\partial Y(n)$. We finally have, for each $n\geq 0$, an inclusion morphism
\begin{displaymath}
  \ii n\colon \DO n\to \OO n.
\end{displaymath}
The set $\setbis{\ii n}{n\in\NN}$ is denoted by $\I$.

A map is a {\em trivial fibration} if it has the right-lifting property with respect to $\I$ and a {\em cofibration} if it has the left-lifting property with respect to all trivial fibrations.

\begin{proposition}\label{prop:factor}
 Any $\infty$-functor $f$ factors as $f=p\circ k$ where $p$ is a trivial fibration and $k$ is a cofibration. 
\end{proposition}
\begin{proof}
  By a standard application of the small object argument, using the fact
  that $\ocat$ is locally presentable.
\end{proof}
On the other hand the maps $\sce n,\tge n\colon n\to n+1$ in the globular category $\glob$ (see Section~\ref{sec:defs}) induce corresponding maps from $\OO n$ to $\OO{n+1}$, of the form $\ii{n+1}\circ\sss n$ and $\ii{n+1}\circ\ttt n$ respectively, where $\sss n,\ttt n\colon \OO n\to\DO{n+1}$. Moreover, we get a pushout diagram
\begin{displaymath}
  \begin{xy}
    \xymatrix{\DO n\ar[r]^{\ii{n}}\ar[d]_{\ii{n}} & \OO n\ar[d]^{\sss{n}}\\
              \OO n\ar[r]_{\ttt{n}} & \DO{n+1}.
}
  \end{xy}
\end{displaymath}
Now the above pushout determines a canonical map
\begin{displaymath}
  \ooo n \colon  \DO{n+1}\to \OO{n}
\end{displaymath}
such that $\ooo n\circ \sss n = \ooo n\circ \ttt n= \id_{\OO{n}}$. Proposition~\ref{prop:factor} applies to $\ooo n$, yielding an object $\PP n$ together with a trivial fibration $\pp n\colon \PP n\to \OO n$ and a cofibration $\kk n\colon \DO{n+1}\to\PP n$ satisfying $\ooo n=\pp n\circ\kk n$. We finally define $\jj n\colon \OO n\to\PP n$ as $\kk n\circ\sss n$ and
\begin{displaymath}
  J=\setbis{\jj n}{n\in\NN}.
\end{displaymath}

\begin{theorem}\label{thm:folkmodel}
  There is a cofibrantly generated model structure on $\ocat$ where $\Weq$ is the class of weak equivalences, $\I$ a set of generating cofibrations and $J$ a set of generating trivial cofibrations.
\end{theorem}
This statement is in fact~\cite[Theorem 4.39]{lafontetal:folkms} and the main result of that article.

\subsection{Path object}\label{subsec:path}

Let $C$ be an object in a model category and $\Delta_C\colon C\to C\times C$ be the diagonal map. A {\em path object} for $C$ consists in an object $P_C$ together with a factorization of $\Delta_C$ of the form
\begin{displaymath}
  \begin{xy}
    \xymatrix{C\ar[r]^{j}\dar[rr]_{\Delta_C} & P_C\ar[r]^{p} & C\times C}
  \end{xy}
\end{displaymath}
where $p$ is a fibration and $j$ is a weak equivalence. Such a $P_C$ is not
unique: in the case of $\ocat$, one particular choice is given by the
functor $\cnx$ we now describe. We first define, by induction on $n$, the
notion of {\em $n$-cylinder between $n$-cells $x$, $y$ of an
$\infty$-category $C$}. A few notations will be useful: for each $n$-cell
$x$ we denote by $\Sce x$, respectively $\Tge x$ its $0$-source $s_0 x$,
respectively $0$-target $t_0 x$. Now let $C$ be an $\infty$\nobreakdash-category and
$x$, $y$ two $0$-cells in it. There is an $\infty$-category $\Hom{C}{x}{y}$
whose $n$-cells are the $(n+1)$\nobreakdash-cells $u$ of $C$ such that $\Sce u=x$ and
$\Tge u=y$. Whenever $u$ is such an $(n+1)$\nobreakdash-cell of $C$, we denote by $\Sht
u$ the corresponding $n$-cell of $\Hom Cxy$. Finally, let $x$, $y$, $z$ be
$0$-cells of $C$. Each $1$-cell $u\colon x\to y$ determines an
$\infty$-functor $-\act u\colon \Hom{C}{y}{z}\to\Hom{C}{x}{z}$ given by
$\Sht v\act u=\Sht{v\comp_0 u}$. Likewise $u\colon y\to z$ determines an
$\infty$-functor $u\act -\colon \Hom Cxy\to\Hom Cxz$ by $u\act\Sht
v=\Sht{u\comp_ 0 v}$.

\begin{definition}\label{def:cylinder}
\begin{enumerate}
\item A {\em 0-cylinder} $U \colon  x \cto y$ in $C$ is given by a
reversible $1$-cell $\Pal U \colon  x \to y$;
\item If $n > 0$, an {\em
$n$-cylinder} $U \colon  x \cto y$ in $C$ is given by two reversible 1-cells
$\Sce U \colon  \Sce x \to \Sce y$ and $\Tge U \colon  \Tge x \to \Tge y$,
together with some $(n-1)$-cylinder $\Sht U \colon  \Tge U\act\Sht x \cto
\Sht y\act\Sce U$ in the $\infty$-category $\HOM{\Sce x}{\Tge y}=\Hom C{\Sce
x}{\Tge y}$.
\end{enumerate}
\end{definition}
If $U \colon  x \cto y$ is an $n$-cylinder in $C$, we write $\Top_C\, U$ and $\Bot_C\, U$ for the $n$-cells $x$ and $y$, or simply $\Top\, U$ and $\Bot\, U$. Figure~\ref{fig:cylone} represents $n$-cylinders for $n=0$ and $n=1$. 
\begin{figure}
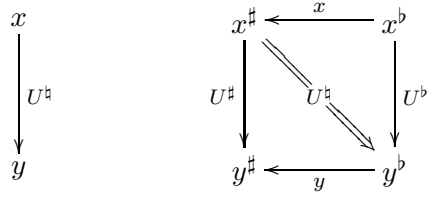

  \centering
  $
\xygraph{ 
!{<0cm,0cm>;<1cm,0cm>:<0cm,1cm>::} 
%
%
!{(0,0) }*+{x}="x"
!{(3,0) }*+{\Tge{x}}="xsharp"
!{(5,0)}*+{\Sce{x}}="xflat"
!{(0,-2) }*+{y}="y"
!{(3,-2) }*+{\Tge{y}}="ysharp"
!{(5,-2)}*+{\Sce{y}}="yflat"
%
%
"x":"y"^{\Pal{U}}
"xflat":"yflat"^{\Sce{U}}
"xflat":"xsharp"_x
"yflat":"ysharp"^y
"xsharp":"ysharp"_{\Tge{U}}
"xsharp":@{=>}"yflat"|{\Pal{U}}
}
$
  \caption{$n$-cylinders for $n=0,1$}
  \label{fig:cylone}
\end{figure}
For each $n\in\NN$, any $(n+1)$-cylinder $W \colon  z \cto z'$ in an $\infty$-category $C$ determines a pair of $n$-cylinders in $C$:
\begin{definition}\label{def:stcyl}
The \emph{source $n$-cylinder} $U \colon  x \cto x'$ and the \emph{target
$n$-cylinder} $V \colon  y \cto y'$ of the $(n+1)$-cylinder $W \colon z
\cto z'$ between $(n+1)$-cells $z \colon  x \to y$ and $z' \colon  x' \to
y'$ are defined inductively by: \begin{itemize}
\item if $n = 0$, then $\Pal U = \Sce W$ and $\Pal V = \Tge W$;
\item if $n > 0$, then $\Sce U = \Sce V = \Sce W$ and $\Tge U = \Tge V =
\Tge W$, whereas the two $(n-1)$\nobreakdash-cylin\-ders $\Sht U$ and $\Sht V$ are
respectively defined as the source and the target of the $n$-cylinder $\Sht
W$ in the $\infty$-category $\HOM {\Sce z} {\Tge{z'}}$.
\end{itemize}
\end{definition}
If $W$ has source $U$ and target $V$ we write $W \colon  U \to V$ or $\Cto {W \colon  U \to V} z {z'}$ (see Figure~\ref{fig:scetge}).
\begin{figure}
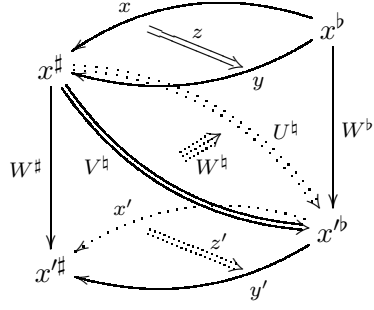

  \centering
  $
\xygraph{ 
!{<0cm,0cm>;<1.5cm,0cm>:<0cm,.9cm>::} 
!{(0,0) }*+{\Tge{x}}="xsharp"
!{(0,-3) }*+{\Tge{x'}}="xsharpprime"
!{(2.5,0.6)}*+{\Sce{x}}="xflat"
!{(2.5,-2.4)}*+{\Sce{x'}}="xflatprime"
!{(.8,0.6)}*+{}="x"
!{(1.7,0)}*+{}="y"
!{(.8,-2.4)}*+{}="xp"  
!{(1.7,-3)}*+{}="yp"
!{(1.1,-1.4)}*+{}="domw"  
!{(1.5,-1)}*+{}="codw"
"xflat":"xflatprime"^{\Sce{W}}
"xsharp":"xsharpprime"_{\Tge{W}}
"xflat":@/^.5cm/"xsharp"^(.3){y}
"xflat":@/_.55cm/"xsharp"_(.7){x}
"xflatprime":@/^.5cm/"xsharpprime"^(.3){y'}
"xflatprime":@{.>}@/_.55cm/"xsharpprime"_(.7){x'}
"xsharp":@{:>}@/^.6cm/"xflatprime"^(.7){\Pal{U}}
"xsharp":@{=>}@/_.6cm/ "xflatprime"_(.3){\Pal{V}}
"x":@{=>}"y"^z
"xp":@{:>}"yp"^(.7){z'}
"domw":@3{.>}"codw"_{\Pal{W}}
}
$
  \caption{source and target of a $2$-cylinder}
  \label{fig:scetge}
\end{figure}
It turns out that the source and target maps so defined satisfy the globular relations, so that the correspondence
\begin{displaymath}
    n\mapsto \setbis{U}{\text{$U$ is an $n$-cylinder in $C$}}
\end{displaymath}
determines a globular set $\Cnx C$. We now turn to {\em trivial $n$-cylinders}:
\begin{definition}
The \emph{trivial $n$-cylinder} $\Triv \, x \colon  x \cto x$ on the $n$-cell $x$ is defined inductively by: 
\begin{itemize}
\item if $n = 0$, then $\Pal{(\Triv \, x)} = \unit x$;
\item if $n > 0$, then $\Sce{(\Triv \, x)} = \unit{\Sce x}$ and $\Tge{(\Triv \, x)} = \unit{\Tge x}$, whereas $\Sht{\Triv \, x}$ is the trivial cylinder $\Triv \Sht x$ in $\HOM {\Sce x} {\Tge x}$.
\end{itemize}
\end{definition}
We write $\Triv_C\, x$ in case we need to mention the ambient $\infty$-category $C$.

Let us finally recall from~\cite[Appendix A]{lafontetal:folkms} that $\Cnx
C$ becomes a strict $\infty$-category when defining units and compositions
as follows:
\begin{definition}
Let  $U \colon  x \cto y$ be an $n$-cylinder. We define the $(n+1)$-cylinder
$\Cto {\unit U \colon \break U \to  U} {\unit x} {\unit y}$ by induction on $n$:
\begin{itemize}
\item if $n = 0$, then $\Sce{(\unit U)} = \Tge{(\unit U)} = \Pal U$, whereas $\Sht{\unit U} = \Triv \Sht{\unit{\Pal U}}$;
\item if $n > 0$, then $\Sce{(\unit U)} = \Sce U$ and $\Tge{(\unit U)} = \Tge U$, whereas $\Sht{\unit U} = \unit {\Sht U}$.
\end{itemize}
\end{definition}

In order to define composition, we first introduce the operation of {\em concatenation}:

\begin{definition}
  Let $\Cto Uxy$ and $\Cto Vyz$ be two $n$-cylinders. The {\em concatenation} $\Cto{V\comp U}{x}{z}$ of $U$ and $V$ is defined by induction on $n$:
  \begin{itemize}
  \item if $n=0$, then $\Pal{(V\comp U)}=\Pal V\Comp 0\Pal{U}$;
  \item if $n>0$, then $\Sce{(V\comp U)}=\Sce V\Comp 0\Sce U$, $\Tge{(V\comp U)}=\Tge V\Comp 0\Tge U$ and $\Sht{V\comp U}=\Sht V\act\Sce U\comp\Tge V\act\Sht U$.
  \end{itemize}
\end{definition}

\begin{definition}
Let $m\geq 1$, $0\leq n<m$ and $\Cto{U}{x}{x'}$, $\Cto{V}{y}{y'}$ two
$m$-cylinders such that $\Glt{n}(U)=\Gls{n}(V)$. The {\em composition}
$\Cto{V \Comp n U}{y\Comp n x}{y'\Comp n x'}$ is defined by induction on $n$
as follows:
\begin{itemize}
\item $\Sce{(V \Comp 0 U)} = \Sce U$, $\Tge{(V \Comp 0 U)} = \Tge V$ and
$\Sht{V \Comp 0 U} = y' \act \Sht U \comp \Sht V \act x$;
\item if $n > 0$, then $\Sce{(V \Comp n U)} = \Sce U = \Sce V$, $\Tge{(U
\Comp n V)} = \Tge U = \Tge V$ and $\Sht{V \Comp n U} = \break\Sht V \Comp{n-1}
\Sht U$.
\end{itemize}
\end{definition}

Note that explicit formulas may be found in~\cite{metayer:respol}. 

For example, Figure~\ref{fig:compo} shows the composition $V\Comp 0 U$ of
two $1$-cylinders $U\colon  x\cto x'$ and $V\colon  y\cto y'$ such that
$\Tge U=\Sce V$. Precisely, the composite $V\Comp 0 U$ is the $1$-cylinder
$W\colon z\cto z'$ where $z=y\Comp 0 x$, $z'=y'\Comp 0 x'$, $\Sce W=\Sce U$,
$\Tge W=\Tge V$, and the $0$-cylinder $\Sht W$ of $\Hom{C}{\Sce
z}{\Tge{z'}}$ is the reversible $1$-cell of $\Hom{C}{\Sce z}{\Tge{z'}}$
given by the following corresponding reversible $2$-cell of $C$:
\begin{displaymath}
\Pal W = (y'\Comp 0 \Pal U)\Comp 1(\Pal V\Comp 0 x). 
\end{displaymath}

\begin{figure}[ht]
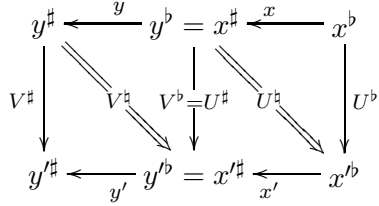

  \centering
$
\xygraph{ 
!{<0cm,0cm>;<1cm,0cm>:<0cm,1cm>::} 
%
%
!{(0,0) }*+{\Tge{y}}="ysharp"
!{(2,0) }*+{\Sce{y}=\Tge{x}}="xsharp"
!{(4,0)}*+{\Sce{x}}="xflat"
!{(0,-2) }*+{\Tge{y'}}="ysharpprime"
!{(2,-2) }*+{\Sce{y'}=\Tge{x'}}="xsharpprime"
!{(4,-2)}*+{\Sce{x'}}="xflatprime"
%
%
"xflat":"xsharp"_{x}
"xsharp":"ysharp"_{y}
"xflatprime":"xsharpprime"^{x'}
"xsharpprime":"ysharpprime"^{y'}
"ysharp":"ysharpprime"_{\Tge{V}}
"xsharp":"xsharpprime"|{\Sce{V}=\Tge{U}}
"xflat":"xflatprime"^{\Sce{U}}
"ysharp":@{=>}"xsharpprime"|{\Pal{V}}
"xsharp":@{=>}"xflatprime"|{\Pal{U}}
}
$
  \caption{composition of $1$-cylinders}
  \label{fig:compo}
\end{figure}
The following result summarizes the main properties of $\cnx$:
\begin{theorem}\label{thm:connect}
The correspondence $C\mapsto\Cnx C$ induces an endofunctor on $\ocat$, and $\Top,\Bot\colon \cnx\to \id$, $\Triv\colon \id\to\cnx$ are natural transformations.
\end{theorem}
An additional property, of particular importance here, is that the functor $\cnx$
preserves \oo-group\-oids: 
\begin{lemma}\label{lemma:cylgrp}
  If $G$ is an $\infty$-groupoid, so is $\Cnx G$.
\end{lemma}
\begin{proof}
 We show, by induction on $n\geq 1$, that if  $G$ is an  $\infty$-groupoid and $\Cto{W\colon U\to V}xy$ is an $n$-cylinder of $G$, there is an $n$-cylinder $\Cto{W'\colon V\to U}{\inv{x}}{\inv{y}}$ such that $W'\Comp{n-1}W=\unit{U}$ and $W\Comp{n-1} W'=\unit{V}$.
 \begin{itemize}
 \item Let $G$ be an $\infty$-groupoid and $\Cto{W\colon U\to V}xy$ a $1$-cylinder of $G$. By definition, we get two $1$-cells $\Pal{U}\colon \Sce x\to\Sce y$, $\Pal{V}\colon \Tge x\to\Tge y$ and a $2$-cell $\Pal{W}\colon \Pal{V}\Comp{0}x\to y\Comp{0}\Pal{U}$ in $G$. Consider $\inv{\Pal{W}}\colon y\Comp{0}\Pal{U}\to \Pal{V}\Comp{0}x$ the $\Comp 1$-inverse of $\Pal W$ and build
\begin{displaymath}
  \inv{y}\Comp{0}\inv{\Pal{W}}\Comp{0}\inv{x}\colon \Pal{U}\Comp{0}\inv{x}\to\inv{y}\Comp{0}\Pal{V}.
\end{displaymath}
If $\Cto{W'\colon V\to U}{\inv{x}}{\inv{y}}$ is the $1$-cylinder of $G$ defined by 
\begin{displaymath}
  \Pal{W'}= \inv{y}\Comp{0}\inv{\Pal{W}}\Comp{0}\inv{x},
\end{displaymath}
we get
\begin{displaymath}
  W'\Comp{0}W=\unit{U}\quad {\mathrm{and}} \quad W\Comp{0}W'=\unit{V},
\end{displaymath}
which proves the case $n=1$.
\item Let $n>1$ and suppose that the property holds for $n-1$. Let $G$ be an
$\infty$-groupoid and $\Cto{W\colon U\to V}{x}{y}$ an $n$-cylinder of $G$.
We get $1$-cells $\Sce{W}\colon \Sce{x}\to\Sce{y}$, $\Tge{W}\colon
\Tge{x}\to\Tge{y}$ and an $(n-1)$-cylinder $\Cto{\Sht{W}\colon
\Sht{U}\to\Sht{V}}{\Tge{W}\act\Sht{x}}{\Sht{y}\act\Sce{W}}$ in $H=\HOM{\Sce
x}{\Tge y}$. Now $H$ is an $\infty$-groupoid, so that the induction
hypothesis applies and there is an $(n-1)$-cylinder in $H$
  \begin{displaymath}
    \Cto{\Sht{W}'\colon \Sht{V}\to\Sht{U}}{\Tge{W}\act\Sht{\inv{x}}}{\Sht{\inv{y}}\act\Sce{W}}
  \end{displaymath}
such that $\Sht{W}'\Comp{n-2}\Sht{W}=\unit{\Sht{U}}$ and $\Sht{W}\Comp{n-2}\Sht{W}'=\unit{\Sht{V}}$. Hence we may define an $n$-cylinder $W'$ of $G$ by $\Sce{W'}=\Sce{W}$, $\Tge{W'}=\Tge{W}$ and $\Sht{W'}=\Sht{W}'$. By construction
\begin{displaymath}
  W'\Comp{n-1}W=\unit{U}\quad {\mathrm{and}} \quad W\Comp{n-1}W'=\unit{V}.
\end{displaymath}
 \end{itemize}
 \end{proof}

\begin{remark}\label{rem:Cyl_oo_n}
Let $n \ge 0$.
The proof of the previous lemma actually shows that if $G$ is a strict $(\infty,
n)$-category, then so is $\Cnx{G}$.
\end{remark}

\subsection{Immersions}

We now introduce a class of morphisms which plays an important part in the proof of the transfer theorem.
\begin{definition}\label{def:immersions}
An $\infty$-functor $f \colon  C \to D$ belongs to the class $\Imm$ of {\em immersions} if and only if there exist $\infty$-functors $g\colon D\to C$ and $h\colon D\to \Cnx D$ satisfying the following properties:
\begin{enumerate}
\item $g$ is a retraction of $f$, that is $g \circ f = \id_C$;
\item $\Top_D \circ h = f \circ g$ and $\Bot_D \circ h = \id_D$;
\item $h \circ f = \Triv_D \circ f$. In other words, $h$ is trivial on $f(C)$.
\end{enumerate}

\end{definition}
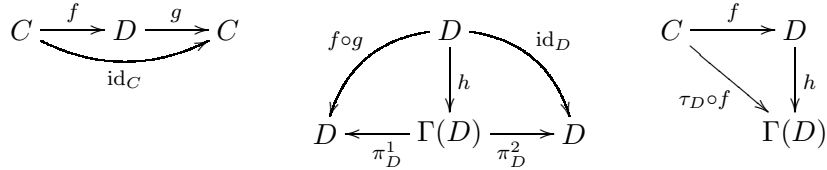
\begin{figure}[ht]
  \centering
  \[
\begin{xy}
\xymatrix{C \ar[r]^{f} \dar[rr]_{\id_C} & D \ar[r]^{g} & C}
\end{xy}
\qquad
\begin{xy}
\xymatrix{ & D \ar[d]^{h}\dar[ld]_{f\circ g} \uar[rd]^{\id_D} & \\ 
D & \Cnx D \ar[l]^{\Top_D} \ar[r]_{\Bot_D} & D}
\end{xy}
\qquad
\begin{xy}
\xymatrix{C \ar[r]^{f} \ar[rd]_{\Triv_D\circ f} & D \ar[d]^{h} \\
 & \Cnx D}
\end{xy}
\]
  \caption{Immersions}
  \label{fig:immersions}
\end{figure}
The following proposition summarizes the properties of immersions we need here.
\begin{proposition}\label{prop:immersions}
  The class $\Imm$ of immersions satisfies the following properties:
  \begin{enumerate}[(i)]
  \item $\Imm$ is closed by pushout; \label{it:pushout}
  \item all trivial cofibrations are immersions; \label{it:trivcofib}
  \item all immersions are weak equivalences. \label{it:immweq}
  \end{enumerate}
\end{proposition}
We refer to~\cite[Section 4.6]{lafontetal:folkms} for the proofs of these statements.

\subsection{Transfer}

Let $\ctg C$, $\ctg D$ be two categories and $L\colon \ctg C\to \ctg D$, $R\colon \ctg D\to \ctg C$ be a pair of functors with $L$ left adjoint to $R$, and suppose that $\ctg C$ is equipped with a model structure. We may define three classes of maps of $\ctg D$ as follows:
\begin{itemize}
\item for each morphism $f$ of $\ctg D$, $f\in\weakeq{\ctg{D}}$ if and only if $R(f)$ is a weak equivalence in $\ctg C$;
\item for each morphism $f$ of $\ctg D$, $f\in\fibr{\ctg{D}}$ if and only if $R(f)$ is a weak equivalence in $\ctg C$;
\item $\cofibr{\ctg{D}}$ is the class of maps having the left-lifting property with respect to $\weakeq{\ctg{D}}\cap\fibr{\ctg{D}}$.
\end{itemize}
We say that $R$ {\em creates} a model structure on $\ctg D$ if $\weakeq{\ctg{D}}$, $\fibr{\ctg{D}}$ and $\cofibr{\ctg{D}}$ are respectively the weak equivalences, fibrations and cofibrations of a model structure on $\ctg D$. Sufficient conditions for this transfer to hold are given by~\cite[Theorem~3.3]{crans:quicms} or~\cite[Proposition~2.3]{beke:shhtwo}. The latter result immediately specializes to the following statement: 
\begin{proposition}\label{prop:transfer}
Let $\ctg C$ a cofibrantly generated model structure, with $I$ a set of generating cofibrations and $J$ a set of generating trivial cofibrations. If $\ctg D$ is locally presentable, then the following conditions are sufficient for $R$ to create a model structure on $\ctg D$: 
\begin{enumerate}[(C1)]
\item the weak equivalences of $\ctg C$ are closed under filtered colimits;
\item $R$ preserves filtered colimits;
\item for each generating trivial cofibration $j$ of $\ctg C$, if $g$ is a pushout of $L(j)$ in $\ctg D$, then $R(g)$ is a weak equivalence of $\ctg C$.
\end{enumerate}
Moreover, if these conditions hold, the model structure so defined is cofibrantly generated and has $L(I)$ as a set of generating cofibrations and $L(J)$ as a set of generating trivial cofibrations.
\end{proposition}

We now turn to the particular case where $\ctg C=\ocat$, $\ctg D=\ogrp$ and $R$ is the inclusion functor $\IN\colon \ogrp\to\ocat$. Note that $\IN$ has a left-adjoint $\FG\colon \ocat\to\ogrp$ building the free $\infty$-groupoid on an $\infty$-category, as well as a right-adjoint $\MG$, building to the maximal $\infty$-groupoid in an $\infty$-category.  Let us first establish a few properties about the adjunction $\FG\dashv \IN$. 

Let $\TT$ be the monad $\IN\FG$ on $\ocat$. Remark that, for any
$\infty$-groupoid $G$, the free $\infty$-groupoid on the underlying
$\infty$-category $\IN(G)$ is naturally isomorphic to $G$ itself. In other
words, the counit $\varepsilon\colon \FG\IN\to 1$ is a natural isomorphism.  It follows that, for any $\infty$-groupoid $G$, we get an isomorphism
\begin{equation}
  \label{eq:isomgr}
  \eta_{\IN(G)}\colon \IN(G)\to \IN\FG\IN(G),
\end{equation}
where $\eta$ denotes the unit of the adjunction. Now, for each $\infty$-category $C$, $\TT(C)$ is of the form $\IN(G)$ where $G$ is an $\infty$-groupoid, and so is $\cnx\TT(C)$ by Lemma~\ref{lemma:cylgrp}, so that
\begin{equation}
  \label{eq:isomcyl}
  \eta_{\cnx\TT(C)}\colon \cnx\TT(C)\to\TT\cnx\TT(C)
\end{equation}
is an isomorphism, as a special case of~(\ref{eq:isomgr}). Thus, we may define a natural transformation
\begin{displaymath}
  \lambda\colon \TT\cnx\to\cnx\TT
\end{displaymath}
by
\begin{equation}
  \label{eq:lambda}
  \lambda_C = \eta^{-1}_{\cnx\TT(C)}\circ\TT\cnx(\eta_C).
\end{equation}
Note also that the monad multiplication $\mu\colon \TT^2\to\TT$ is also a natural isomorphism, and we get
\begin{equation}
  \label{eq:monadone}
  \TT(\eta_C) = \eta_{\TT(C)} = \mu^{-1}_C.
\end{equation}
We may now state the following result:
\begin{lemma}\label{lemma:immpres}
  The monad $T$ preserves immersions.
\end{lemma}

\begin{proof}
  Let $f\colon C\to D$ be an immersion, and $f'=\TT(f)$. By Definition~\ref{def:immersions}, there are $g\colon D\to C$ and $h\colon D\to\Cnx D$ such that
\begin{eqnarray}
  g\circ f & = & \id_C ;
\label{eq:gfid}\\
  \Top_D\circ h & = & f\circ g ;
\label{eq:tophfg} \\
   \Bot_D\circ h & = & \id_D ;
\label{eq:bothid}\\
   h\circ f & = & \Triv_D\circ f.
\label{eq:hftrivf}
\end{eqnarray}
Let $g'=\TT(g)$ and $h'=\lambda_D\circ\TT(h)$. We need to establish the following equations:
\begin{eqnarray}
  g'\circ f' & = & \id_{\TT(C)} ;
   \label{eq:gfid'}\\
  \Top_{\TT(D)}\circ h'& = & f'\circ g';
   \label{eq:tophfg'}\\
  \Bot_{\TT(D)}\circ h' & = & \id_{\TT(D)} ;
   \label{eq:bothid'}\\
  h'\circ f' & = & \Triv_{\TT(D)}\circ f'.
    \label{eq:hftrivf'}
\end{eqnarray}
Equation~(\ref{eq:gfid'}) is just functoriality. Let us prove~(\ref{eq:tophfg'}). First remark that $\Top$ is a natural transformation, so that the following diagram commutes:
\begin{equation}
\label{eq:naturdiag}
  \begin{xy}
    \xymatrix{\cnx(D)\ar[d]_{\cnx(\eta_D)}\ar[r]^{\Top_D} & D\ar[d]^{\eta_D} \\
              \cnx\TT(D)\ar[r]_{\Top_{\TT(D)}} & \TT(D).}
  \end{xy}
\end{equation}
We may now build the following commutative diagram:
\begin{equation}
  \label{eq:bigdiag}
  \begin{xy}
    \xymatrix{\TT\cnx(D)\ar[r]^{\TT(\Top_D)}\ar[d]^{\TT\cnx(\eta_D)}\Dar[dd]_{\lambda_D} & \TT(D)\ar[d]_{\TT(\eta_D)}\Uar[dd]^{\id_{\TT(D)}}\\
              \TT\cnx\TT(D)\ar[r]^{\TT(\Top_{\TT(D)})}\ar[d]^{\eta^{-1}_{\cnx\TT(D)}} & \TT^2(D)\ar[d]_{\eta^{-1}_{\TT(D)}}\\
              \cnx\TT(D)\ar[r]_{\Top_{\TT(D)}}  &  \TT(D).}
  \end{xy}
  \end{equation}
In fact the upper square is the image of~(\ref{eq:naturdiag}) by $\TT$ and the lower square commutes by naturality of $\eta$. Hence
\begin{eqnarray*}
  \Top_{\TT(D)}\circ h' & = & \Top_{\TT(D)}\circ\lambda_D\circ\TT(h)\\
                       & = & \TT(\Top_D)\circ\TT(h)\\
                       & = & \TT(\Top_D\circ h)\\
                       & = & \TT(f\circ g)\\
                       & = & f'\circ g'
\end{eqnarray*}
 which gives~(\ref{eq:tophfg'}). Likewise, we get the following commutative
 diagram:
 \begin{equation}
   \label{eq:smalldiag}
   \begin{xy}
     \xymatrix{\TT(D)\ar[r]^{\TT(h)}\ar[rd]_{\id_{\TT(D)}}\Uar[rr]^{h'} &
     \TT\cnx(D)\ar[r]^{\lambda_D}\ar[d]|{\TT(\Bot_D)} &
     \cnx\TT(D)\ar[ld]^{\Bot_{\TT(D)}},\\
                      & \TT(D) &              }
   \end{xy}
 \end{equation}
where the left hand triangle commutes by applying $\TT$ to~(\ref{eq:bothid}), and the right hand triangle commutes by replacing $\Top$ with $\Bot$ in~(\ref{eq:bigdiag}). Hence $\Bot_{\TT(D)}\circ h'=\id_{\TT(D)}$ and~(\ref{eq:bothid'}) is proved. Finally, by using the naturality of $\Triv$ instead of $\Top$, we get a commutative diagram analogue to~(\ref{eq:bigdiag}):
\begin{equation}
  \label{eq:bigdiag'}
  \begin{xy}
    \xymatrix{\TT(D)\ar[r]^{\TT(\Triv_D)}\ar[d]^{\TT(\eta_D)}\Dar[dd]_{\id_{\TT(D)}} & \TT\cnx(D)\ar[d]_{\TT\cnx(\eta_D)}\Uar[dd]^{\lambda_D}\\
              \TT^2(D)\ar[r]^{\TT(\Triv_{\TT(D)})}\ar[d]^{\eta^{-1}_{\TT(D)}} & \TT\cnx\TT(D)\ar[d]_{\eta^{-1}_{\cnx\TT(D)}}\\
            \TT(D)\ar[r]_{\Triv_{\TT(D)}}   & \cnx\TT(D).
}
  \end{xy}
\end{equation}
Hence
\begin{eqnarray*}
  h'\circ f' & = & \lambda_D\circ\TT(h)\circ\TT(f)\\
             & = & \lambda_D\circ\TT(h\circ f)\\
             & = & \lambda_D\circ\TT(\Triv_D\circ f)\\
             & = & \lambda_D\circ\TT(\Triv_D)\circ f'\\
             & = & \Triv_{\TT(D)}\circ f'
\end{eqnarray*}
which gives~(\ref{eq:hftrivf'}) and ends the proof.
\end{proof}

\begin{lemma}\label{lemma:pushout}
 Let $f\colon C\to D$ be an immersion, and suppose that the following square is a pushout in $\ogrp$:
 \begin{displaymath}
   \begin{xy}
     \xymatrix{\FG C\ar[r]^u\ar[d]_{\FG(f)} & G\ar[d]^g\\
               \FG D\ar[r]_{v} & H.}
   \end{xy}
 \end{displaymath}
Then $\IN(g)$ is an immersion.  
\end{lemma}
\begin{proof}
  As $\IN$ is left adjoint to $\MG$, it preserves pushouts, so that the following square is a pushout in $\ocat$:
\begin{displaymath}
   \begin{xy}
     \xymatrix{\TT C\ar[r]^{\IN(u)}\ar[d]_{\TT(f)} & \IN G\ar[d]^{\IN(g)}\\
               \TT D\ar[r]_{\IN(v)} & \IN H.}
   \end{xy}
 \end{displaymath}
By Lemma~\ref{lemma:immpres}, $\TT(f)$ is an immersion, and so is its pushout $\IN(g)$, by Proposition~\ref{prop:immersions}(\ref{it:pushout}).
\end{proof}

\begin{lemma}\label{lemma:cthree}
 If $j$ is a generating trivial cofibration of $\ocat$, and $g$ is a pushout of $\FG(j)$ in $\ogrp$, then $\IN(g)$ is a weak equivalence of $\ocat$. 
\end{lemma}
\begin{proof}
Let $j$ be a generating trivial cofibration of $\ocat$, and $g$ be a pushout of $\FG(j)$ in $\ogrp$. By Proposition~\ref{prop:immersions}(\ref{it:trivcofib}), $j$ is an immersion, and so is $\IN(g)$, by Lemma~\ref{lemma:pushout}. By Proposition~\ref{prop:immersions}(\ref{it:immweq}), $\IN(g)$ is a weak equivalence.  
\end{proof}

We may finally state the main result of this section:
\begin{theorem}\label{thm:transfer}
The forgetful functor $\IN\colon \ogrp\to\ocat$ creates a model structure
on $\ogrp$ in which the weak equivalences are the morphisms $f$ such that
$\IN(f)$ belongs to $\Weq$. Moreover, the model structure so defined has
$(\FG(\ii{k}))_{k\in\NN}$ as a family of generating cofibrations, and
$(\FG(\jj{k}))_{k\in\NN}$ as a family of generating trivial cofibrations.
\end{theorem}
\begin{proof}
As the model structure on $\ocat$ is cofibrantly generated and $\ogrp$ is
locally presentable, Proposition~\ref{prop:transfer} applies, and it
suffices to check conditions (C1), (C2) and (C3).
Condition (C1) is proved in~\cite{lafontetal:folkms}, and condition 
(C2) follows from the fact that $\IN$ has a right-adjoint $\MG$, hence
preserves colimits, and in particular filtered ones. Condition (C3) is
Lemma~\ref{lemma:cthree}. The statement about generating families follows
from Proposition~\ref{prop:transfer}.  \end{proof}

\begin{remark}
Using Remark~\ref{rem:Cyl_oo_n},
one can easily adapt the proof of the previous theorem to show that a
similar theorem holds for strict $(\infty, n)$-categories. In particular,
the inclusion functor $\wncat \to \wcat$ creates a model structure on
$\wncat$.
\end{remark}

We call the model structure just defined the {\em folk model structure} on \oo-groupoids. We denote its weak equivalences by $\Wfolk$ and its trivial fibrations by $\TFfolk$. Note that a morphism  $f$ is in $\TFfolk$ if and only if $\IN(f)$ is a trivial fibration of $\ocat$. 

\begin{proposition}\label{prop:folktriv}
A morphism $f\colon G\to H$ of \oo-groupoids belongs to $\TFfolk$ if and
only if the following conditions are satisfied:
\begin{enumerate}
\item for every object $y$ of $H$, there exists an object $x$ of $G$ such
that $f(x) = y$;
\item for all $n \ge 1$ and every pair $(u, v)$ of parallel $(n-1)$-arrows of
$G$, the map
\[ G(u, v)_0 \to H(f(u), f(v))_0 \]
is surjective.
\end{enumerate} 
\end{proposition}

\begin{proof}
By definition, $f$ belongs to $\TFfolk$ if and only if $U(f)$ has the right
lifting property with respect to $\I$. This proposition is then just a matter of
translation.
\end{proof}

\section{Comparison}\label{sec:compar}

In this section, we show that the folk model structure on strict
\oo-groupoids defined in the previous section coincides with the
Brown-Golasi\'nski model structure. To see this, it suffices to prove
that they have the same weak equivalences and the same trivial fibrations.

\begin{proposition}\label{prop:same_w_eq}
We have
$\Wgr = \Wfolk = \Wcc$.
\end{proposition}

\begin{proof}
We first show that $\Wfolk = \Wgr$.
In a strict \oo-groupoid, two $n$-arrows $f$ and $g$ are $\omega$-equivalent if and
only if there exists an $(n+1)$-arrow $a \colon f \to g$, that is if and
only if $f$
and $g$ are homotopic. Therefore a morphism of strict \oo-groupoids is in $\Wfolk$
if and only if it satisfies condition \ref{item:folk_w_eq} of Proposition
\ref{prop:def_w_eq}. The statement is thus exactly the equivalence between
conditions \ref{item:w_eq} and \ref{item:folk_w_eq} of this very proposition.

By Proposition \ref{prop:cc_w_eq}, we have $\Wcc = \Wgr$, hence the result.
\end{proof}

\begin{proposition}\label{prop:same_tr_fib}
We have $\TFfolk = \TFcc$.
\end{proposition}

\begin{proof}
To prove the equivalence between the two notions of trivial fibrations, we
will use the descriptions of these notions provided by
Propositions~\ref{prop:cc_tr_fib} and~\ref{prop:folktriv}. The
conditions for being in $\TFfolk$ are a priori stronger. Let $f \colon G \to
H$ be a in $\TFcc$. Let us prove it is actually in $\TFfolk$. There is
nothing to prove for the conditions in dimension $0$ and $1$. Let $n \ge 2$
and let $u, v$ be two parallel $(n-1)$-arrows. We want to show that the map
\[ \Homens{G}{u}{v} \to \Homens{H}{f(u)}{f(v)} \]
is surjective. Let $b$ be an $n$-arrow from $f(u)$ to $f(v)$ in $H$. 
Set $x = \Gls{0}(u)$. Then 
$b' = \Glid{\Glw{0}(f(u))} \comp_0 b$ is an $n$-arrow of $H$
from $\Glid{f(x)}$ to 
$\Glw{0}(f(u)) \comp_0 f(v)$. Since the map
\[
\Homens{G}{\Glid{x}}{\Glw{0}(u) \comp_0 v}
\to
\Homens{H}{f(\Glid{x})}{f(\Glw{0}(u) \comp_0 v)}
\]
is surjective, there exists an $n$-arrow $a'$ of $G$ from
$\Glid{x}$ to $\Glw{0}(u) \comp_0 v$ such that $f(a') =
b'$. Then, the $n$-arrow $a =  \Glid{u} \comp_0 a'$ is from
$u$ to $v$ and we have
\[
\begin{split}
f(a) & = f(\Glid{u} \comp_0 a')\\
     & = \Glid{f(u)} \comp_0 b' \\
     & = \Glid{f(u)} \comp_0 \Glid{\Glw{0}(f(u))} \comp_0 b  \\
     & = b.
\end{split}
\]
\end{proof}

\begin{theorem}
The Brown-Golasi\'nski model structure and the folk model structure on strict
\oo-groupoids coincide.
\end{theorem}

\begin{proof}
By the two previous propositions, these model structures have the same weak
equivalences and the same trivial fibrations.
\end{proof}

\bibliographystyle{abbrv}

\bigskip

\textsc{\footnotesize Dimitri Ara, Institut Math\'ematiques de Jussieu,
Universit\'e Paris Diderot~-- Paris 7, Case 7012, B\^atiment Chevaleret,
75205 Paris Cedex 13, France}

\footnotesize{\emph{E-mail address}: \url{ara@math.jussieu.fr}}

\bigskip

\textsc{\footnotesize Fran\c{c}ois M\'etayer, Laboratoire PPS, Universit\'e
Paris Diderot -- Paris 7 \&  CNRS, Case 7014, B\^atiment Chevaleret, 75205
Paris Cedex 13, France}

\footnotesize{\emph{E-mail address}: \url{metayer@pps.jussieu.fr}}

\end{document}